\documentclass{article}

\usepackage{etoolbox}

\makeatletter
\patchcmd{\maketitle}{\@fnsymbol}{\@arabic}{}{}
\makeatother

\title{Reverse Quantile-RK and its Application to Quantile-RK}
\author{Emeric Battaglia\footnote{\url{ebattagl@uci.edu}, corresponding author}, Anna Ma\\\\ \textit{Department of Mathematics, University of California, Irvine, CA 92697 USA}}

\date{\today}

\usepackage[lmargin=1in,rmargin=1in,tmargin=1in,bmargin=1in]{geometry}
\usepackage[section]{placeins}
\usepackage[font=small]{caption}
\usepackage[T1]{fontenc}
\usepackage{
	amsmath,           
    amsthm,
    mathtools,
	amssymb,           
	enumerate,		   
	graphicx,          
	lastpage,          
	multicol,          
	multirow,          
    diagbox,
    mathrsfs,
    thmtools,
    thm-restate,
    scrextend,
    algorithm,
    algpseudocode,
    subcaption,
    fancyhdr,
    bbm,
    tikz,
    pgfplots,
    textcomp,
    charter,
    ulem,
    changepage,
    stmaryrd,
    hyperref,
    xurl,
    eucal
}

\newtheorem{theorem}{\textit{Theorem}}

\newtheorem{lem}{\textit{Lemma}}

\newtheorem{cor}{\textit{Corollary}}

\newcommand{\groupp}[1]{\l(#1\r)}
\newcommand{\groupb}[1]{\l[#1\r]}
\newcommand{\groupc}[1]{\l\{#1\r\}}
\newcommand{\groupabs}[1]{\l|#1\r|}

\renewcommand{\l}{\left}
\renewcommand{\r}{\right}

\renewcommand{\a}{\hat{a}}
\newcommand{\E}{\mathbb{E}}
\newcommand{\N}{\mathbb{N}}
\newcommand{\PP}{\mathbb{P}}

\newcommand{\R}{\mathbb{R}}
\newcommand{\x}{{x^\ast}}

\newcommand{\s}{\sigma}

\newcommand{\ve}{\varepsilon}

\newcommand{\sset}{\subseteq}

\newcommand{\sm}{\setminus}

\newcommand{\inv}{^{-1}}
\newcommand{\dl}{^*}

\newcommand{\ra}{\rightarrow}
\renewcommand{\lll}{\langle}
\newcommand{\rrr}{\rangle}

\renewcommand{\leq}{\leqslant}
\renewcommand{\geq}{\geqslant}
\newcommand{\Def}{\vcentcolon =}

\newcommand{\Mod}[1]{\ (\mathrm{mod}\ #1)}

\DeclareMathOperator*{\argmax}{arg\,max}

\DeclareMathOperator{\quant}{-quant}

\usetikzlibrary{matrix}
\usepgfplotslibrary{groupplots}
\pgfplotsset{compat=newest}

\allowdisplaybreaks

\begin{document}
\maketitle

\begin{abstract}
    When solving linear systems $Ax=b$, $A$ and $b$ are given, but the measurements $b$ often contain corruptions. Inspired by recent work on the quantile-randomized Kaczmarz method, we propose an acceleration of the randomized Kaczmarz method using quantile information. We show that the proposed acceleration converges faster than the randomized Kaczmarz algorithm. In addition, we show that our proposed approach can be used in conjunction with the quantile-randomized Kaczamrz algorithm, without adding additional computational complexity, to produce both a fast and robust iterative method for solving large, sparsely corrupted linear systems. Our extensive experimental results support the use of the revised algorithm. 
\end{abstract}

\centerline{\noindent\textbf{Key Words:} Kaczmarz method, quantile methods, randomized iterative methods, sparse corruption}

\section{Introduction}

    In signal and image processing, the recovery of the solution $\x$ to linear a system of equations 
    \begin{equation}
        Ax=b, \label{eq:syseq}
    \end{equation}
    for given $A,b$ is of particular interest. In light of modern computational advances that allow the collection of vast amounts of data, research steered towards a renewed interest in the iterative methods such as the Kaczmarz method \cite{kacz}, alternatively known as \texttt{ART} \cite{ART}. Iterates of the Kaczmarz algorithm are given by projecting the last iterate onto a hyperplane $\groupc{x:\lll x,a_i\rrr=b_i}$ where $i$ is chosen deterministically at each iteration and $a_i$ is the $i$-th row of $A$. However, providing general convergence estimates for this method has proven to be a challenge. This is partly due to the convergence rate being affected by the order in which the aforementioned hyperplanes are chosen. Strohmer and Vershynin introduced a randomized variant, which we refer to as \texttt{RK}, of this method that overcomes this issue \cite{RK}. Namely, choosing the indices $i$ randomly with probability proportional to $\|a_i\|^2$ at each iteration allowed them to prove exponential convergence of \texttt{RK}. Heuristically, we can argue that choosing the row index at random allows us to say \texttt{RK} is expected to converge at the same rate independent of the order of the rows of $A$. Moreover, \texttt{RK} is readily applicable to large-scale systems where only a few rows of the matrix might be known or loaded into memory at any given time. Perhaps unsurprisingly, this has spawned further research of variations of \texttt{RK} that may accelerate convergence under appropriate conditions; for examples, see \cite{pavedwithgoodintentions,greedyavgblocksparse, greedworks, gRK, gRK2,SKM}. 
    
    The exponential convergence to the solution of the system (\ref{eq:syseq}) requires the system to be consistent. In practice, this may not be the case. When the system is inconsistent, we see that \texttt{RK} does not converge to the typically sought-after least-squares solution. More specifically, if (\ref{eq:syseq}) has a least squares solution $\hat{x}$ and we apply \texttt{RK}, the iterates $x_k$ converge to $\hat{x}$ up to an error horizon dependent on the magnitude of corruption of $b$~\cite{needell_LS}. Similarly, \cite{doublynoisysystem} proves convergence up to an error horizon when there is noise in both $A$ and $b$. If, instead, a least-squares solution is desirable, Zouzias and Freris \cite{Zouzias_REK} offers an adaptation of \texttt{RK} called the randomized extended Kaczmarz method, in which iterates are were shown to converge in expectation to the least-norm least-squares solution to (\ref{eq:syseq}). In place of seeking the least-norm least-squares solution, it might be more important to be able to find a sparse least-squares solution. To this end, \cite{sparseREK} successfully proves expected linear convergence to a sparse least-squares solution using another variation of \texttt{RK}. 
    
    Now suppose that the right-hand side $b$ has a number of corruptions: If we write $b=b_t+b_\ve$ where $b_t$ is the true measurement and $b_\ve$ is the corruption introduced to $b_t$ to attain the empirical value $b$, it is a natural question to ask whether there is still hope to recover the solution $\x$ to $Ax=b_t$. Previous work that considers noise on the right-hand side indicates that the iterates can only be guaranteed to converge to within a ball of the solution. However, in the case that $b_\ve$ is sparse, \cite{qRK} and \cite{qRK2} show a modification of \texttt{RK} witnesses convergence to $\x$ provided $A$ contains enough redundant data that information is still salvageable despite possible adversarial corruption. More specifically, the modification involves computing the $q$-th quantile $Q$ of the absolute value of the entries of the residual $Ax_k-b$ and only allows rows corresponding to residual entries less than or equal to $Q$ in magnitude to be chosen. However, this strategy for overcoming corruption introduces a slower convergence rate. In this work, we propose an accelerated Kaczmarz algorithm that can be used in conjunction with \texttt{qRK} to produce a fast yet robust iterative method for solving sparsely corrupted linear systems.
    
    \subsection{Notation}
      Before presenting relevant background, we first define useful notation. Let $[n]$ denote the set $\{1,\dots,n\}$. We use $\|\cdot\|_F$ to denote the Frobenius norm, $\|\cdot\|$ to denote the Euclidean norm, and $\|x\|_0$ to denote the number of non-zero entries of $x$. For $A\in \R^{m\times n}$ and $S\sset [m]$, we will use $A_S \in \mathbb{R}^{|S| \times n}$ to denote the submatrix of $A$ given by the rows of $A$ indexed by $S$. When referencing a single row of $A$, we will write $a_i$ in place of $a_{\{i\}}$, and we define $\a_i\Def \frac{a_i}{\|a_i\|}$ to be the normalized $i$-th row $A$. We use $\beta$ to express the fraction of entries of $b$ that are corrupt. In the event $\beta>0$, we decompose $b=b_t+b_\ve$ where $b_t$ is the true right hand side of the system of equations and $b_\ve$ is nonzero except in the entries where $b$ is corrupted. We use $\x$ to denote the solution to the system (\ref{eq:syseq}), where $b$ is replaced with $b_t$ if $\beta>0$. Finally, for any finite multiset $S$, we let $q\quant(S)$ denote the $q$-th quantile of $S$. We assume that $q|S|$ is always integer valued. Because $S$ is a multiset, it is possible that the multiset $L=\{s: s\leq q\quant(S)\}$ has more than $q|S|$ entries, so we adopt the convention that $L$ is constructed using the first $q|S|$ entries satisfying $s\leq q\quant(S)$. We also let $\{s:s>q\quant(S)\}$ be the multiset constructed by removing the elements of $L$ from $S$ according to multiplicity. If $S$ can be indexed by a set $I\sset \N$, i.e. $S=\{s_i\}_{i\in I}$, we adopt the convention that $L\dl = \{i\in I: s_i \leq q\quant(S)\}$ is constructed using the first $q|S|$ entries satisfying $s_i\leq q\quant(S)$, and $\{i\in I:s_i>q\quant(S)\}$ is constructed by removing the elements of $S$ indexed by $L\dl$ from $S$.
    
    \subsection{Background}
        The Kaczmarz method is a deterministic iterative algorithm designed to approximate a solution $\x\in \R^n$ to $Ax=b$ when $A\in \R^{m\times n}$ and $b\in\R^m$ are known \cite{kacz}. Given some initial $x_0\in\R^n$, the iterates are defined as
        \begin{equation}
            x_{k+1}=x_k+\frac{b_i-\lll x_k,a_i\rrr}{\|a_i\|^2}a_i. \label{eq:iterate_update}
        \end{equation}
        That is, $x_{k+1}$ is the projection of $x_k$ onto the hyperplane $\{x\in\R^n:\lll a_i,x\rrr=b_i\}$ where $a_i$ is the $i$-th row of $A$, $b_i$ is the $i$-th entry of $b$, and $i= k+1 \Mod{m}$. Strohmer and Vershynin showed that if $A$ is over-determined and of full rank, and (\ref{eq:syseq}) is a consistent system, choosing the index $i$ with probability $\frac{\|a_i\|^2}{\|A\|_F^2}$ implies the approximates $x_k$ converge to the true solution $\x$ and the expected error can be bounded by: 
        \begin{equation}
            \E\groupb{\|x_k-\x\|^2}\leq \groupp{1-\frac{\s_{\min}^2(A)}{\|A\|_F^2}}^k\|x_0-\x\|^2,
        \end{equation}
        where $\s_{\min}(A)$ is the smallest singular value of $A$. We will write \texttt{RK} to denote the randomized Kaczmarz method just outlined, i.e. (\ref{eq:iterate_update}) with $i$ chosen randomly with probability proportional to row norms \cite{RK}. 
        
        If the system is inconsistent as a result of corruption in $b$, a modification of \texttt{RK} proposed by \cite{qRK2} may have convergence results provided the conditioning of the matrix $A$ in relation to the support of the error is sufficiently well-behaved. In the presence of spares corruption in $b$, a modification of \texttt{RK} was proposed and shown to converge to $\x$ provided the conditioning of $A$ is sufficiently well-behaved \cite{qRK2, qRK}. More precisely, \cite{qRK2} and \cite{qRK} examine the case where the system is given by $Ax=b$ where $b=b_t+b_\ve$, $Ax=b_t$ is consistent, $A$ is a row-normalized, over-determined, full-rank matrix, and $\|b_\ve\|_0\leq \beta m$ for a choice of $\beta\in[0,1]$. This modification of the \texttt{RK} algorithm, which is known as quantile-based randomized Kaczmarz algorithm (\texttt{qRK}), is outlined in Algorithm \ref{alg:qrk}.
        \begin{algorithm}[H]
            \caption{Quantile-based Randomized Kaczmarz Method (\texttt{qRK})}
            \begin{algorithmic}
                \State \textbf{Inputs: } $A,\, b,\, q,\, x_0,\, K$ 
                \State $k \gets 0$
                \For{$k<K$} 
                    \State $Q \gets q\quant\groupp{\groupc{\{\frac{|\lll x_k,a_j\rrr - b_j|}{\|a_j\|}:j\in [m]}}$ 
                    \State $S \gets \groupc{j\in [m]: \frac{|\lll x_k,a_j\rrr - b_j|}{\|a_j\|}\leq Q }$
                    \State Select $i\in S$ with probability $\frac{\|a_i\|^2}{\|A_S\|_F^2}$
                    \State $x_{k+1} \gets x_k + \frac{\groupp{b_i - \lll x_k , a_i\rrr }}{\|a_i\|}a_i$
                    \State $k\gets k+1$
                \EndFor
            \end{algorithmic}
            \label{alg:qrk}
        \end{algorithm} 
        
        The convergence rate of Algorithm \ref{alg:qrk} is a function of the smallest singular value of submatrices of $A$:
        \begin{equation}
            \s_{\alpha,\min}(A)=\min_{\substack{S\sset [m] \\ |S|=\alpha m}}\inf_{\|x\|=1}\|A_S x\|,
        \end{equation}
        which was introduced in~\cite{qRK} and is a slight modification of the term introduced in~\cite{qRK2}. This term is relevant in bounding $\sum_{i\in I} \groupabs{\lll x_k,a_i \rrr-b_i}^2$ for $I\sset [m]$ at each iteration $k$. Namely, when it is known $I$ does not include any of the indices of rows that might have corruption in $b$, it follows that 
        \[
            \sum_{i\in I} \groupabs{\lll x_k,a_i \rrr-b_i}^2 = \|A_I x_k - A_I \x\|^2 \geq \s_{\min}^2(A_I)\|x_k-\x\|^2.
        \]
        Because it is not necessarily known at each iteration $k$ which rows are not corrupt and are permissible to choose from, the best one can hope for is to provide a uniform bound on $\s_{\min}(A_I)$ for all $I\sset [m]$ of the some fixed cardinality. This is exactly the purpose $\s_{\alpha,\min}(A)$ serves. For the results below to be substantial, we assume that the matrices we are dealing with are conditioned well enough that removing some rows does not reduce the smallest singular value of the subsystem too much. Some matrices have a specific structure that might make them particularly susceptible to this poor condition. For example, empirically, sparse matrices commonly fail to have a large enough $\s_{\alpha,\min}$ term or are otherwise not sufficiently well-conditioned; this can be thought of as the matrix not having enough data to make a sizeable portion of the rows redundant. This numerical behavior is presented in Section \ref{sec:final_rem} to demonstrate cases when Algorithms \ref{alg:qrk} and \ref{alg:dqrk} might struggle or outright fail.
        
        As was done with \texttt{RK}, we will refer to Algorithm \ref{alg:qrk} as \texttt{qRK}. Other works have presented variations of \texttt{qRK} that allow for a noticeable increase in convergence rate or make it more versatile. We refer to the reader to \cite{subsamplingqRK, timevaryingcorruptionqRK, blockaccelqRK} for improvements on \texttt{qRK}. 
    
\section{Contributions}
    \label{sec:contributions}
    When choosing a row to use in the projection step at each iteration of \texttt{RK}, intuition suggests that projecting onto the furthest hyperplane would yield the largest progress toward the solution vector $\x$. Indeed, choosing the row $i=\argmax_{j\in [m]}\groupc{\groupabs{\lll x_k,a_j\rrr -b_j}}$ at iteration $k$ is the principle design of the Motzkin method \cite{MM}. 

    To still allow for an element of randomness, we introduce Algorithm \ref{alg:rqrk}, which randomly chooses from indices $i$ where 
    \begin{equation}
        \frac{|\lll x_k,a_i\rrr - b_i|}{\|a_i\|} \label{eq:modified_res}
    \end{equation}
    exceeds a quantile threshold. Though computing the quantile increases the computational cost, this serves to accelerate \texttt{RK} as Algorithm \ref{alg:rqrk} is forced to choose indices $i$ corresponding to larger values of (\ref{eq:modified_res}), which is the distance between successive iterates of $x_k$. 

    \begin{algorithm}[H]
        \caption{Reverse Quantile-Based Randomized Kaczmarz Method (\texttt{rqRK})}
        \begin{algorithmic}
            \State \textbf{Inputs: } $A,\, b,\, q,\, x_0,\, K$
            \State $k \gets 0$
            \For{$k<K$}
                \State $Q \gets q\quant\groupp{\groupc{\frac{|\lll x_k,a_j\rrr - b_j|}{\|a_j\|}:j\in [m]}}$ 
                \State $S \gets \groupc{j\in [m]: \frac{|\lll x_k,a_j\rrr - b_j|}{\|a_j\|}> Q }$
                \State Select $i\in S$ with probability $\frac{\|a_i\|^2}{\|A_S\|_F^2}$ 
                \State $x_{k+1} \gets x_k + \frac{b_i - \lll x_k , a_i\rrr }{\|a_i\|^2}a_i$
                \State $k\gets k+1$
            \EndFor
        \end{algorithmic}
        \label{alg:rqrk}
    \end{algorithm}
    Because Algorithm \ref{alg:rqrk} ``reverses'' the direction of the inequality in the choice of $S$ of Algorithm~\ref{alg:qrk}; we will call this the reverse quantile-based randomized Kaczmarz method, or \texttt{rqRK} for short. 
    
    Our main theoretical result guarantees the convergence of Algorithm \ref{alg:rqrk}, which we highlight is an acceleration of \texttt{RK}. We state this result here:
    \begin{restatable}[Main Result]{theorem}{mainresult}
        Assume $A$ is an over-determined, full-rank, and $\x$ satisfies $A\x=b$. Suppose $x_0$ satisfies $\langle x_0,a_i\rangle=b_i$ for some $i\in[m]$. Then \texttt{rqRK} with parameter $\frac{1}{m}\leq q \leq \frac{m-1}{m}$ yields
        \[
            \E\groupb{\|x_k-\x\|^2}\leq \groupp{1-\frac{\s_{\min}^2(A)}{\|A\|_F^2} -\frac{\s_{q,\min}^2(A)}{qm}\frac{\min_{1\leq j\leq m}\groupc{\frac{\|a_j\|^2}{\|A\|_F^2}}}{\max_{1\leq j \leq m}\{\|a_j\|^2\}} }^k\|x_0-\x\|^2.
        \]
        In particular, for row-normalized matrices, this simplifies to
        \[
            \E\groupb{\|x_k-\x\|^2}\leq \groupp{1-\frac{\s_{\min}^2(A)}{m}-\frac{\s_{q,\min}^2(A)}{qm^2}}^k\|x_0-\x\|^2.
        \]
        \label{theorem:rqrk}
    \end{restatable}

    The hypothesis requiring $x_0$ to lie on some solution hyperplane is not, to the best of the authors' knowledge, present in other works surrounding \texttt{RK}. This is a minor detail that ensures at least one entry of the residual $Ax_1-b$ is 0. The importance of this, and how to bypass it with a slight modification to the statement of Theorem \ref{theorem:rqrk}, is explained at the end of Theorem \ref{theorem:rqrk} in Section \ref{sec:rqrkanalysis}. 

    When $qm<n$, we have $\s_{q,\min}(A)=0$, so our bound degrades back to the original \texttt{RK} bound. We also assume $q\geq \frac{1}{m}$ because otherwise \texttt{rqRK} degrades to \texttt{RK}, where again, the Strohmer and Vershynin bound holds \cite{RK}. More nuances of the result are discussed at the end of Section \ref{sec:rqrkanalysis}. 
    
    The issue with always choosing rows that are ``most violated" as in \texttt{rqRK} is that such a scheme leads to a minimal tolerance of corruption; in practice, the $b$ term in equation (\ref{eq:syseq}) may have corruption. It stands to reason that the indices $i$ corresponding to corrupted $b_i$ entries would typically yield among the largest of the terms $\groupc{\groupabs{\lll x_k,a_j\rrr -b_j}/\|a_j\|}_{j=1}^m$ when $x_k$ is close to $\x$; this would lead \texttt{rqRK}, and similarly the Motzkin method, to often project iterates $x_k$ onto a hyperplane stemming from a corrupted row, i.e. away from our desired solution $\x$. To address this, \texttt{qRK} uses an upper quantile to reduce the likelihood of choosing a corrupted index. Embedding Motzkin's method into \texttt{qRK} by electing to project onto the furthest hyperplane indexed by the admissible indices $S$ can still be fatal to the convergence of the iterates. Namely, it is possible that an index for a corrupted row $i$ might still lie in the set of admissible indices $S$, and if corrupted rows are the most violated, we may find this method chooses a corrupted row too frequently. 
    
    We introduce a second quantile to avoid this issue while improving the convergence rate of \texttt{qRK}. This second quantile allows us to choose rows that will project $x_k$ further on average than typical in \texttt{qRK} while still selecting indices randomly enough that corrupted indices do not appear often enough to inhibit convergence. This is the basis for Algorithm \ref{alg:dqrk}, which serves to accelerate \texttt{qRK} with minimal increased computational cost.        
    
    Based on the discussion on how we might accelerate the convergence of \texttt{qRK}, we formally outline Algorithm \ref{alg:dqrk}:
    \begin{algorithm}[H]
        \caption{Double Quantile-Based Random Kaczmarz Method (\texttt{dqRK})}
        \begin{algorithmic}
            \State \textbf{Inputs: } $A,\, b,\, q_0,\, q_1,\, x_0,\, K$
            \State $k \gets 0$
            \For{$k<K$}
                \State $Q_0 \gets q_0\quant\groupp{\groupc{\frac{|\lll x_k,a_j\rrr - b_j|}{\|a_j\|}:j\in [m]}}$
                \State $Q_1 \gets q_1\quant\groupp{\groupc{\frac{|\lll x_k,a_j\rrr - b_j|}{\|a_j\|}:j\in [m]}}$
                \State $S \gets \groupc{j\in [m]: Q_0< \frac{|\lll x_k,a_j\rrr - b_j|}{\|a_j\|} \leq Q_1 }$
                \State Select $i\in S$ with probability $\frac{\|a_i\|^2}{\|A_{S}\|_F^2}$
                \State $x_{k+1} \gets x_k + \frac{b_i - \lll x_k , a_i\rrr }{\|a_i\|^2}a_i$
                \State $k\gets k+1$
            \EndFor
        \end{algorithmic}
        \label{alg:dqrk}
    \end{algorithm}
    It is now clear that elements of \texttt{qRK} and \texttt{rqRK} are married to procure Algorithm \ref{alg:dqrk}. Namely, the upper quantile acts as in \texttt{qRK} to control the impact of corruptions, and the lower quantile component acts as in \texttt{rqRK} to, on average, increase the distance between iterates, thereby increasing the convergence rate. In light of there being two quantiles utilized, we call Algorithm \ref{alg:dqrk} the double quantile-based randomized Kaczmarz algorithm, or \texttt{dqRK} for short. Given both \texttt{qRK} and \texttt{rqRK} play a role in creating \texttt{dqRK}, it may come as no surprise that characteristics of the convergence results from both are emulated in the convergence results of \texttt{dqRK}:
    
    \begin{restatable}{theorem}{dqrkresult}
        Suppose $A$ is row-normalized, over-determined, and full-rank. Let $\x$ be a solution to the system $Ax=b_t$ and $b=b_t+b_\ve$ where $\|b_\ve\|_0\leq \beta m$. Suppose $\beta < q_0 < q_1 < 1-\beta$, $q_1-q_0>\beta$, and 
        \[
            \frac{q_1}{q_1-q_0-\beta}\groupp{\frac{2\sqrt{\beta}}{\sqrt{1-q_1-\beta}}+\frac{\beta}{1-q_1-\beta}} < \frac{1}{\s_{\max}^2(A)}\groupp{\s_{q_1-\beta,\min}^2(A)+\frac{\s_{q_0-\beta,\min}^2(A)}{q_0m}}.
        \]
        If $x_0$ is chosen such that $\langle x_0, a_i\rangle = b_i$ for some $i\in [m]$, then the expected error of the \texttt{dqRK} algorithm decays as
        \[
            \E\|x_{k}-\x\|^2\leq (1-C)^k \|x_0-\x\|^2,
        \]
        where 
        \[
            C = \groupp{q_1-q_0-\beta}\groupp{\frac{\s_{q_1-\beta,\min}^2(A)}{(q_1-q_0)q_1 m}+\frac{\s_{q_0-\beta,\min}^2(A)}{(q_1-q_0)q_0q_1 m^2}} - \frac{\s_{\max}^2(A)}{(q_1-q_0)m}\groupp{\frac{2\sqrt{\beta}}{\sqrt{1-q_1-\beta}}+\frac{\beta}{1-q_1-\beta}}.
        \]
        \label{theorem:dqrk}
    \end{restatable}

    The reader is encouraged to note the similarity between Theorem \ref{theorem:dqrk} and the analog convergence results for \texttt{qRK}:
    \begin{theorem}
        [\cite{qRK}, Main Theorem]
        Suppose $A$ is row-normalized, over-determined, and full-rank. Let $\x$ be the solution to the system $Ax=b_t$ and $b=b_t+b_\ve$ where $\|b_\ve\|_0\leq \beta m$. Suppose $\beta < q<1-\beta$, and 
        \[
            \frac{q}{q-\beta}\groupp{\frac{2\sqrt{\beta}}{\sqrt{1-q-\beta}}+\frac{\beta}{1-q-\beta}}<\frac{\s_{q-\beta,\min}(A)^2}{\s_{\max}^2(A)}.
        \]
        Then the expected error of the \texttt{qRK} algorithm decays as
        \[
            \E\groupb{\|x_k-\x\|^2}\leq (1-C)^k\|x_0-\x\|^2,
        \]
        where
        \[
            C=(q-\beta)\frac{\s_{q-\beta,\min}(A)^2}{q^2m}-\frac{\s_{\max}^2(A)}{qm}\groupp{\frac{2\sqrt{\beta}}{\sqrt{1-q-\beta}}+\frac{\beta}{1-q-\beta}}.
        \]
    \end{theorem}

\section{Analysis of the Reverse-Quantile Random Kaczmarz Method}
    \label{sec:rqrkanalysis}
    In providing analysis of the convergence of \texttt{rqRK}, we begin Section \ref{sec:rqrkanalysis} by providing two self-contained results in Section \ref{sec:usefulresults} useful to characterizing the bound on the expectation of $\|x_k-\x\|^2$ under the \texttt{rqRK} scheme based on known results about \texttt{RK}. In Section \ref{sec:mainresult}, these results will be used to provide a concrete bound on the expected error at the $k$-th iteration of \texttt{rqRK} only using a bound on the expected error of the $k$-th iteration of \texttt{RK} and some information given by the quantile parameter. 
    
    \subsection{Useful Results}
        \label{sec:usefulresults}
        The following useful results extend beyond the scope of random numerical linear algebra; in fact, they stand alone as general probability results.
        \begin{lem}
            Let $X$ and $Y$ be random variables that assume values from $\{1,\dots,n\}$ and $\{\ell+1,\dots,n\}$, respectively, and $\PP(X=k)\neq 0$ for all $k\in[n]$. Furthermore, suppose that $Y$ assumes the value $k$ with probability $q_k = \PP(X=k\mid X\neq 1,\dots,\ell)$. Let $f:[n]\ra \R$ be non-decreasing. Then \[\E[f(Y)]\geq \E[f(X)]+(f(\ell+1)-f(\ell))\groupp{\sum_{k=1}^\ell \PP(X=k)}.\]
            Moreover, if there are at least two distinct elements in the range of $f$,
            \[
                \E[f(Y)]>\E[f(X)].
            \]
            \label{lem:auxlemma}
        \end{lem}
        
        \begin{proof}
            To simplify notation, let $p_k=\PP(X=k)$ and let $x_k = f(k)$. Then $q_k=\frac{\PP(X=k,X\neq 1,\dots, \ell)}{\PP(X\neq 1,\dots, \ell)}=\frac{p_k}{1-(p_1+\cdots+p_\ell)}$ for $k>\ell$. Then we have 
            \begin{align}
                \E[f(Y)]-\E[f(X)] &= \sum_{k=\ell+1}^n q_k x_k -\sum_{k=1}^n p_k x_k \nonumber \\
                            &= \frac{1}{1-(p_1+\cdots+p_\ell)} \sum_{k=\ell+1}^n p_k x_k - \sum_{k=1}^n p_k x_k \nonumber \\
                            &= \groupp{\frac{1}{1-(p_1+\cdots+p_\ell)}-1}\sum_{k=\ell+1}^n p_k x_k - \sum_{k=1}^\ell p_kx_k \nonumber \\
                            &\geq x_{\ell+1} \groupp{\frac{1}{1-(p_1+\cdots+p_\ell)}-1}\sum_{k=\ell+1}^n p_k - x_\ell \groupp{\sum_{k=1}^\ell p_k}   \label{eq:ineq} \\
                            &= x_{\ell+1} \groupp{\sum_{k=1}^\ell p_k} -x_\ell \groupp{\sum_{k=1}^\ell p_k} \nonumber \\
                            &=(x_{\ell+1}-x_\ell)\groupp{\sum_{k=1}^\ell p_k}. \nonumber
            \end{align}  
            Now consider the case where there are at least two distinct elements in the range of $f$. Let $k$ be some index for which $x_k<x_{k+1}$. If $k=\ell$, then we trivially have $\E[f(Y)]-\E[f(X)]>0$ from the last equality. If $k<\ell$, the first inequality \eqref{eq:ineq} above is strict because $x_j<x_\ell$ for some $j<\ell$ and the right summand contains $x_j$. If $k>\ell$, then the same inequality is strict because $x_{\ell+1}<x_{k+1}$ and the left summand contains the term $x_{k+1}$.
        \end{proof}
        We can provide a more useful bound by repeatedly applying Lemma \ref{lem:auxlemma} in the case that $\ell=1$.
        \begin{cor}
            Given $X$, $Y$, and $f$ as in Lemma \ref{lem:auxlemma},  
            \[
                \E[f(Y)] \geq \E[f(X)] + \sum_{k=1}^{\ell} \Big(f(k+1)-f(k)\Big)\PP\groupp{X=k\mid X\not\in [k-1]},
            \]
            with the convention that $\PP(X=k\mid X\not\in [k-1])$ for $k=1$ is $\PP(X=1)$. 
            \label{cor:maincorollary}
        \end{cor}
        \begin{proof}
            We use the earlier convention where $x_k=f(k)$. We also define $Z_1\Def X$ and let $Z_j$ assume values $[n]\sm\{1,\dots, j-1\}$  with probability $\PP(Z_j=k)=\PP(Z_{j-1}=k\mid Z_{j-1}\neq j-1)$. In this case, we can apply the lemma to each pair $Z_{j+1}$ and $Z_{j}$ to see $\E[Z_{j+1}]-\E[Z_{j}]\geq (x_{j+1}-x_{j})\PP(Z_{j}=j)$. It is not difficult to see $Y$ and $Z_{\ell+1}$ are identically distributed, so we have
            \begin{align*}
                \E[f(Y)]-\E[f(X)] &= \E[f(Y)]-\E[f(Z_\ell)]+\sum_{j=1}^{\ell-1}\E[f(Z_{j+1})]-\E[f(Z_{j})] \\
                            &= \sum_{j=1}^{\ell}\E[f(Z_{j+1})]-\E[f(Z_{j})] \\
                            &\geq \sum_{j=1}^{\ell}(x_{j+1}-x_j)\PP(Z_j=j).
            \end{align*}
        \end{proof}
    
    \subsection{The Main Result}
        \label{sec:mainresult}
        Here, we provide some notations that are useful for the proof of Theorem \ref{theorem:rqrk}. We define $X_k$ to be the random variable that assumes index $i$ with probability $\frac{\|a_i\|^2}{\|A\|_F^2}$. Similarly we define $Y_k$ to be the random variable that assumes index $i\in S$ with probability $\frac{\|a_i\|^2}{\|A_S\|_F^2}$ where $S$ is given by 
        \[
            S=\Big\{j\in [m]: \groupabs{\lll x_{k-1},a_j\rrr-b_j}>q\quant\big(\groupc{\groupabs{\lll x_{k-1},a_j\rrr-b_j}:j\in [m]}\big)\Big\}.
        \]
        In other words, $X_k$ and $Y_k$ choose index at iteration $k$ in accordance with \texttt{RK} and \texttt{rqRK}, respectively. The error $x_k-\x$ is also dependent on whether we use \texttt{RK} or \texttt{rqRK} to choose our next row, so we will let $e_{X,k}$, $e_{Y,k}$ denote the error random variable of our $k$-th iteration when \texttt{RK}, \texttt{rqRK}, respectively, is used to choose our row. Lastly, we will use $\E_k$ to denote the expectation given that the first $k-1$ chosen rows are fixed. If we need to specify the first $k-1$ rows chosen were $s_1,\dots, s_{k-1}$, we will then denote the expectation by $\E_{(s_1,\dots,s_{k-1})}$. In the case that the previous rows were fixed and it does not matter whether we used \texttt{RK} or \texttt{rqRK}, we will simply use $e_{k}$. We remind the reader that $\a_i$ denotes $\frac{a_i}{\|a_i\|}$. For convenience, we restate the main theorem here before proving it. 
        \mainresult*
        
        \begin{proof}
            We first note that as in Strohmer and Vershynin's \cite{RK} work, we have 
            \[
                \|x_k-\x\|^2=\groupp{1-\l|\l\lll \frac{x_{k-1}-\x}{\|x_{k-1}-\x\|}, \a_{Y_k}\r\rrr\r|^2}\|x_{k-1}-\x\|^2.
            \]
            This equality is obtained using Pythagorean's theorem and the definition of the iterates $x_k$. In particular, $\x-x_k$ and $x_{k-1}-x_k$ are orthogonal, so 
            \[
                \|x_k-\x\|^2 = \|x_{k-1}-\x\|^2-\|x_{k-1}-x_{k}\|^2,
            \]
            and, since $A\x=b$,
            \begin{align*}
                \|x_{k}-x_{k-1}\|^2&=\left\| x_{k-1} + \frac{b_{Y_k} - \lll x_{k-1} , a_{Y_k}\rrr }{\|a_{Y_k}\|^2}a_{Y_k}-x_{k-1}\right\|^2  \\
                &=\left\|\lll \x-x_{k-1}, \a_{Y_k}\rrr \a_{Y_k}\right\|^2\\
                &=\groupabs{\lll \x-x_{k-1}, \a_{Y_k}\rrr }^2.
            \end{align*}
            As such, using the \texttt{rqRK} scheme, we have
            \begin{equation}
                \E_k\groupb{\|x_k-\x\|^2}=\groupp{1-\E_k\groupb{\l|\l\lll \frac{x_{k-1}-\x}{\|x_{k-1}-\x\|}, \a_{Y_k}\r\rrr\r|^2}}\|x_{k-1}-\x\|^2. \label{eq:thm1eq1}
            \end{equation}
            In light of this, we will focus our efforts on bounding 
            \[
                \E_k\groupb{\l|\l\lll \frac{e_{k-1}}{\|e_{k-1}\|}, \a_{Y_k}\r\rrr\r|^2}.
            \]
            To do so, we will use Corollary~\ref{cor:maincorollary}. Let $f_k= g_k\circ \s$, where 
            \[
                g_k(j)=\l|\l\lll \frac{e_{k-1}}{\|e_{k-1}\|}, \a_j\r\rrr\r|^2,
            \]
            and $\s$ is a permutation of $[m]$ such that $g_k(\s(1)),\dots, g_k(\s(m))$ is non-decreasing. Then, by Corollary~\ref{cor:maincorollary},
            \begin{equation}
                \E_k\groupb{\l|\l\lll \frac{e_{k-1}}{\|e_{k-1}\|}, \a_{Y_k}\r\rrr\r|^2} \geq \E_k\groupb{\l|\l\lll \frac{e_{k-1}}{\|e_{k-1}\|}, \a_{X_k}\r\rrr\r|^2}  + \sum_{j=1}^{qm}\Big(f_k(j+1)-f_k(j)\Big)\PP\groupp{X_k = j\mid X_k\not\in[j-1]}, \label{eq:corollary_ineq}
            \end{equation}
            with the convention that $\PP(X_k = j\mid X_k\not\in [j-1])$ for $j=1$ is $\PP(X_k = 1)$. For the first term on the right-hand side of the inequality (\ref{eq:corollary_ineq}), we note that
            \begin{equation}
                \E_k\groupb{\l|\l\lll \frac{e_{k-1}}{\|e_{k-1}\|}, \a_{X_k}\r\rrr\r|^2} = \sum_{j=1}^m \frac{\|a_j\|^2}{\|A\|_F^2}\groupabs{\l\lll \frac{e_{k-1}}{\|e_{k-1}\|}, \a_j\r\rrr}^2 \geq \frac{\s_{\min}^2(A)}{\|A\|_F^2}.
                \label{eq:RKbound}
            \end{equation}
            Now consider the second term on the right-hand side of (\ref{eq:corollary_ineq}). We let 
            \[
                Q=q\quant\groupp{\groupc{\frac{1}{\|a_j\|}\l|\l\lll x_{k-1},a_j\r\rrr-b_j\r|}_{j=1}^m}.
            \]
            Since the system is consistent, $\l|\l\lll x_{k-1},a_j\r\rrr-b_j\r|=|\lll e_{k-1},a_j\rrr|$. Letting $i \in [m]$ be such that $f_k(qm+1)=\l|\l\lll \frac{e_{k-1}}{\|e_{k-1}\|}, \a_i\r\rrr\r|^2$, we have 
            \begin{equation}
                f_k(qm+1)=\frac{1}{\|e_{k-1}\|^2\|a_i\|^2}\l|\l\lll e_{k-1}, a_i\r\rrr\r|^2\geq \frac{Q^2}{\|e_{k-1}\|^2}.
                \label{eq:fkbound}
            \end{equation}
            Moreover, $\l|\l\lll \frac{e_{k-1}}{\|e_{k-1}\|}, \a_j\r\rrr\r|=0$ for some $j$ because the $(k-1)$-th iterate was chosen as a projection of the $(k-2)$-th iterate onto some hyperplane of the form $\{x:\lll x, a_j\rrr=b_j\}$. This ensures $f_k(1)=0$. Letting $N=\groupc{j:\l|\l\lll e_{k-1}, \a_j\r\rrr\r|\leq Q}$, by  choice of $Q$, it is clear $|N|=qm$ and we have 
            \begin{equation}
                \begin{split}
                qm Q^2 &\geq \sum_{j\in N} \l|\l\lll e_{k-1}, \a_j\r\rrr\r|^2 \\
                &\geq \frac{1}{\max_{1\leq j \leq m}\{\|a_j\|^2\}}\sum_{j\in N} \l|\l\lll e_{k-1}, a_j \r\rrr\r|^2 \\
                &\geq \frac{\s_{\min}^2 (A_N) \|e_{k-1}\|^2}{\max_{1\leq j \leq m}\{\|a_j\|^2\}}\\
                &\geq \frac{\s_{q,\min}^2(A)\|e_{k-1}\|^2}{\max_{1\leq j \leq m}\{\|a_j\|^2\}}.
                \end{split}
                \label{eq:qmQ2bound}
            \end{equation}

            We arrive at the following:
            \begin{align}
                \E_k\groupb{\l|\l\lll \frac{e_{k-1}}{\|e_{k-1}\|}, \a_{Y_k}\r\rrr\r|^2} &\geq \E_k\groupb{\l|\l\lll \frac{e_{k-1}}{\|e_{k-1}\|}, \a_{X_k}\r\rrr\r|^2} + \left(\sum_{j=1}^{qm}f_k(j+1)-f_k(j)\right)\min_{1\leq j\leq m}\groupc{\frac{\|a_j\|^2}{\|A\|_F^2}} \label{eq:hypothesis_note}\\
                &\geq \frac{\s_{\min}^2(A)}{\|A\|_F^2} +\frac{Q^2}{\|e_{k-1}\|^2} \min_{1\leq j\leq m}\groupc{\frac{\|a_j\|^2}{\|A\|_F^2}} \nonumber \\
                &\geq \frac{\s_{\min}^2(A)}{\|A\|_F^2} +\frac{\s_{q,\min}^2(A)}{qm}\frac{\min_{1\leq j\leq m}\groupc{\frac{\|a_j\|^2}{\|A\|_F^2}}}{\max_{1\leq j \leq m}\{\|a_j\|^2\}}, \nonumber
            \end{align}
            where the first inequality follows from (\ref{eq:corollary_ineq}) and the bound $\PP(X_k = j \mid X_k\not\in[j-1])\geq \PP(X_k=j)=\|a_j\|^2/\|A\|_F^2$, the second follows from (\ref{eq:RKbound}) and (\ref{eq:fkbound}), and the last follows from (\ref{eq:qmQ2bound}).

            Consequently, 
            \[
                \E_k\groupb{\|x_k-\x\|^2}\leq \groupp{1-\frac{\s_{\min}^2(A)}{\|A\|_F^2} -\frac{\s_{q,\min}^2(A)}{qm}\frac{\min_{1\leq j\leq m}\groupc{\frac{\|a_j\|^2}{\|A\|_F^2}}}{\max_{1\leq j \leq m}\{\|a_j\|^2\}} }\|x_{k-1}-\x\|^2.
            \]
            Temporarily letting 
            \[
                \alpha=\groupp{1-\frac{\s_{\min}^2(A)}{\|A\|_F^2} -\frac{\s_{q,\min}^2(A)}{qm}\frac{\min_{1\leq j\leq m}\groupc{\frac{\|a_j\|^2}{\|A\|_F^2}}}{\max_{1\leq j \leq m}\{\|a_j\|^2\}} },
            \]
            we see
            \begin{align*}
                \E\groupb{\|e_{Y,k}\|^2} &\leq \sum \E_{(n_1,\dots,n_{k-1})}\groupb{\|e_{Y,k}\|^2} \PP\big((Y_1,\dots, Y_{k-1})=(n_1,\dots,n_{k-1})\big) \\
                &\leq \alpha \sum \|e_{k-1}\|^2 \PP\big((Y_1,\dots, Y_{k-1})=(n_1,\dots,n_{k-1})\big) \\
                &= \alpha \E\groupb{\|e_{Y,k-1}\|^2},
            \end{align*}
            where the sum is taken over all possible combinations $(n_1,\dots,n_{k-1})$ of indices of rows. It is then clear by induction
            \[
                \E\groupb{\|e_{Y,k}\|^2}< \groupp{1-\frac{\s_{\min}^2(A)}{\|A\|_F^2} -\frac{\s_{q,\min}^2(A)}{qm}\frac{\min_{1\leq j \leq m}\groupc{\frac{\|a_j\|^2}{\|A\|_F^2}}}{\max_{1\leq j \leq m}\{\|a_j\|^2\}} }^k \|e_0\|^2.
            \]
        \end{proof}

        If we do not assume $x_0$ lies on a solution hyperplane, then the right hand side of (\ref{eq:hypothesis_note}) may only be bounded below by 
        \[
            \frac{\s_{\min}^2(A)}{\|A\|_F^2}
        \]
        when $k=1$. This is because $f_1(qm+1)$ may equal $f_1(1)$ if $x_0$ is not necessarily on a solution hyperplane. We can artificially construct a scenario where this happens by letting $A$ be a full-rank row-normalized matrix and choosing $x_0$ to be a point equidistant from all the solution hyperplanes. If we remove the hypothesis that $x_0$ lies on a solution hyperplane, it is easy to see our new result becomes 
        \[
            \E\groupb{\|x_k-\x\|^2}\leq \groupp{1-\frac{\s_{\min}^2(A)}{\|A\|_F^2} -\frac{\s_{q,\min}^2(A)}{qm}\frac{\min_{1\leq j\leq m}\groupc{\frac{\|a_j\|^2}{\|A\|_F^2}}}{\max_{1\leq j \leq m}\{\|a_j\|^2\}} }^{k-1}\groupp{1-\frac{\s_{\min}^2(A)}{\|A\|_F^2}}\|x_0-\x\|^2.
        \]
        In particular, for row-normalized matrices, this simplifies to
        \[
            \E\groupb{\|x_k-\x\|^2}\leq \groupp{1-\frac{\s_{\min}^2(A)}{m}-\frac{\s_{q,\min}^2(A)}{qm^2}}^{k-1}\groupp{1-\frac{\s_{\min}^2(A)}{m}}\|x_0-\x\|^2.
        \]
        
        As noted after the statement of Theorem \ref{theorem:rqrk} in Section \ref{sec:contributions}, this bound is not an improvement on that of \texttt{RK} in the case that $qm<n$ since $\sigma_{q,\min}(A) = 0$. However, in this case, we still see from Lemma \ref{lem:auxlemma} that as long as $x_{k-1}\neq \x$, we have 
        \[
            \E_k\groupb{\l|\l\lll \frac{e_{k-1}}{\|e_{k-1}\|}, \a_{Y_k}\r\rrr\r|^2} > \E_k\groupb{\l|\l\lll \frac{e_{k-1}}{\|e_{k-1}\|}, \a_{X_k}\r\rrr\r|^2}.
        \]
        Note this expectation is taken assuming the first $k-1$ elected rows are fixed. Using another collection of $k-1$ rows, it might be the case that $x_{k-1}=\x$. In any case, as long as there is one collection $s_1,\dots,s_{k-1}$ rows for which $x_{s_{k-1}}\neq \x$, we have \[\E_{(s_1,\dots,s_{k-1})}\groupb{\l|\l\lll \frac{e_{k-1}}{\|e_{k-1}\|}, \a_{Y_k}\r\rrr\r|^2} > \E_{(s_1,\dots,s_{k-1})}\groupb{\l|\l\lll \frac{e_{k-1}}{\|e_{k-1}\|}, \a_{X_k}\r\rrr\r|^2},\] and thus,
        \begin{align*}
            \E\groupb{\|e_{Y,k}\|^2} &\leq \sum \E_{(s_1,\dots,s_{k-1})}\groupb{\|e_{Y,k}\|^2} \PP\big((Y_1,\dots, Y_{k-1})=(s_1,\dots,s_{k-1})\big) \\
                &< \sum \E_{(s_1,\dots,s_{k-1})}\groupb{\|e_{X,k}\|^2} \PP\big((Y_1,\dots, Y_{k-1})=(s_1,\dots,s_{k-1})\big) \\
                &\leq \sum \groupp{1-\frac{\s_{\min}^2(A)}{\|A\|_F^2}}\|e_{k-1}\|^2 \PP\big((Y_1,\dots, Y_{k-1})=(s_1,\dots,s_{k-1})\big)\\
                &=\groupp{1-\frac{\s_{\min}^2(A)}{\|A\|_F^2}}\E\groupb{\|e_{Y,k-1}\|^2},
        \end{align*}
        where the sum is taken over all possible combinations of rows $s_1,\dots, s_{k-1}$. Then we have 
        \[
            \E\groupb{\|e_{Y,k}\|^2}<\groupp{1-\frac{\s_{\min}^2(A)}{\|A\|_F^2}}^k\|e_0\|^2.
        \]
        Perhaps more importantly, this makes it clear that at the $k$-th step with the same $x_{k-1}$ value, \texttt{rqRK} is expected to perform strictly better than \texttt{RK}.
        
        In setting the parameter $q$ to the largest possible value $(m-1)/m$ forces \texttt{rqRK} to choose the index corresponding to the largest element of the set $\{|\lll x_k,a_j\rrr - b_j|/\|a_j\|:j\in [m]\}$. That is, \texttt{rqRK} is forced to choose the index of the largest residual entry. This is exactly the procedure for the Motzkin method as presented by Haddock and Ma ~\cite{greedworks}. It is also noted in ~\cite{greedworks} that the previous best worst-case decay factor $\alpha$ satisfying $\E\|x_{k+1}-\x\|^2\leq \alpha \|x_k-\x\|^2$ was known to be 
        \[
            \alpha_1\Def1-\frac{\s_{\min}^2(A)}{m}.
        \]
        Our work, however, shows that $\alpha_1$ can be reduced to 
        \[
            \alpha_2\Def1-\frac{\s_{\min}^2(A)}{m}-\frac{\s_{1-1/m,\min}(A)^2}{m^2-m}.
        \]
        This is the same term given by choosing $k=m$ in equation (\ref{eq:kres}) at the end of Section \ref{sec:convergenceofdqrk}. For remarkably ill-conditioned matrices, this new decay factor $\alpha_2$ may not differ from $\alpha_1$. However, $\alpha_2<\alpha_1$ if and only if $\s_{1-1/m,\min}(A)>0$. In the case that $m>n$, this last condition is equivalent to requiring that any submatrix of $A$ given by removing a single row is still of full rank.

\section{Analysis of the Double-Quantile Random Kaczmarz Method}
    \label{sec:convergenceofdqrk}
    In this section, we prove the convergence result of \texttt{dqRK}. It will shortly become evident that this result is a product of melding our \texttt{rqRK} results with Steinerberger's work on \texttt{qRK} \cite{qRK}. To this end, within Section \ref{sec:prereq} we present useful corollaries to~\cite[Lemma~1,2]{qRK}, written in terms of the parameters of \texttt{dqRK} provided in Section \ref{sec:contributions}. We then present Lemma \ref{lem:lemdqrk}, which serves as the \texttt{dqRK} analog to \cite[Lemma~3]{qRK}. In Section \ref{sec:boundondqrk}, we mesh Corollaries \ref{cor:Q1boundcor} and \ref{cor:expectedcorrupterrbound} and Lemma \ref{lem:lemdqrk} to bound the expected approximation error of \texttt{dqRK}. For the rest of the paper, we assume $A$ is row-normalized. 

    \subsection{Prerequisites}
        \label{sec:prereq}
        The following will be used to present our results. As \texttt{dqRK} has two quantiles to track, we denote $Q_0,Q_1$ to be the $q_0,q_1$ quantile value of $\big\{|\lll x_k,a_j\rrr - b_j|:j\in [m]\big\}$. We will also use $C$ to denote the set of indices corresponding to corrupted rows. We will let $B_1$ be the set of indices $j\in [m]$ for which $|\lll x_k,a_j\rrr-b_j|\leq Q_1$, and $B$ will be the set of indices $j\in [m]$ for which $Q_0<|\lll x_k,a_j\rrr-b_j|\leq Q_1$. We view $B$ and $B_1$ as the set of admissible indices for which \texttt{dqRK} and \texttt{qRK}, respectively, randomly choose from. Two particular subsets of $C$ are also of importance: $S$ will be the set of indices $i\in C\cap B$ and $S_0$ will be the set of indices $i\in C \cap (B_1\sm B)$. With this notation, $S$ and $S_0\cup S$ are the collections of corrupted indices \texttt{dqRK} and \texttt{qRK}, respectively, might select. For convenience, these freshly introduced terms are enumerated here:
        \begin{align*}
            Q_0&=q_0\quant\groupp{\big\{|\lll x_k,a_j\rrr - b_j|:j\in [m]\big\}} \\
            Q_1&=q_1\quant\groupp{\big\{|\lll x_k,a_j\rrr - b_j|:j\in [m]\big\}} \\
            C &= \groupc{j\in [m]: (b_\ve)_j \neq 0}\\
            S_0 &= \groupc{j\in C: |\lll x_k,a_j\rrr-b_j|\leq Q_0} \\
            S\,\, &= \groupc{j\in C: Q_0< |\lll x_k,a_j\rrr-b_j|\leq Q_1} \\
            B_1 &= \groupc{j\in [m]: |\lll x_k,a_j\rrr-b_j|\leq Q_1} \\
            B \,\, &= \groupc{j\in [m]: Q_0< |\lll x_k,a_j\rrr-b_j|\leq Q_1}.
        \end{align*}
        Lastly, we write expectation over a set $X$ as $\E_{j\in X}Z=\frac{1}{|X|}\sum_{j\in X}Z(j)$.
        
        As it stands, Corollaries \ref{cor:Q1boundcor} and \ref{cor:expectedcorrupterrbound} follow directly from \cite[Lemma~1,2]{qRK}. Corollary \ref{cor:Q1boundcor} will serve to bound the upper quantile $Q_1$. This is then implemented in Corollary \ref{cor:expectedcorrupterrbound}, which serves as a means to moderate the magnitude of deviation from the desired solution that might be incurred by selecting a violated row. 
        Lemma \ref{lem:lemdqrk}, however, requires slight modification of \cite[Lemma~3]{qRK}. This is an artifact of the subsystem examined in its proof that now depends on both the upper and the lower quantile of \texttt{dqRK} as opposed to solely the upper quantile in \texttt{qRK}. In fact, this is where we implement the techniques used to prove the results of \texttt{rqRK} to recover a similar bound. As such, we present Corollaries \ref{cor:Q1boundcor} and \ref{cor:expectedcorrupterrbound} in terms of the parameters for \texttt{dqRK} without proof and then offer the necessary modifications to obtain Lemma \ref{lem:lemdqrk}. 
        
        Unless otherwise stated, vectors $x_{k}$ are iterates of \texttt{dqRK} and we take expectation conditional on the first $k-1$ iterates.  
        
        \begin{cor} 
            Let $0<q_0<q_1<1-\beta$, suppose $A\x=b_t$ for some $b_t\in \R^m$, and let $b_\ve\in \R^m$ with $\|b_\ve\|_0\leq \beta m$. If $x_k\in \R^n$ is the the $k$-th iterate of \texttt{dqRK}, we can bound the $q_1$ quantile of residuals $Q_1$ as follows:
            \[
                Q_1\leq \frac{\s_{\max}(A)}{\sqrt{m}\sqrt{1-q_1-\beta}}\|x_k-\x\|.
            \]
            \label{cor:Q1boundcor}
        \end{cor}
        Aside from notational differences, Corollary \ref{cor:Q1boundcor} follows directly from \cite[Lemma~1]{qRK}. This is because even in the \texttt{dqRK} setting, we know there are at least $(1-q_1-\beta)m$ \textit{non-corrupt} residual values with magnitude larger than $Q_1$. This allows us to similarly conclude 
        \[
            (1-q_1-\beta) Q_1^2 \leq \sum_{j\not\in C} \groupabs{b_j-\lll x_k,a_j\rrr}^2 \leq \s_{\max}^2(A) \|x_k-\x\|^2.       
        \]
        
        \begin{cor}
            In the same setting as Corollary \ref{cor:Q1boundcor}, we have
            \begin{equation}
                \E_{i\in S}\|x_{k+1}-\x\|^2\leq \groupp{1+\frac{\s_{\max}^2(A)}{\sqrt{|S|}\sqrt{m}}\groupp{\frac{2}{\sqrt{1-q_1-\beta}}+\frac{\sqrt{\beta}}{1-q_1-\beta}}}\|x_k-\x\|^2,\label{eq:expectedcorrupterrboundqrk}
            \end{equation}
            where $S$ is the set of admissible corrupt indices. 
            \label{cor:expectedcorrupterrbound}
        \end{cor}
        
        Our set $S$ of corrupt indices admissible in \texttt{dqRK} depends on both $Q_0$ and $Q_1$. The analog to $S$ in \cite{qRK}, i.e. the set of corrupt indices admissible in \texttt{qRK}, only depends on $Q_1$. As such it may seem improbable that the right-hand side of (\ref{eq:expectedcorrupterrboundqrk}) only depends on $q_1$. However, the proof in \cite{qRK} hinges only on the quantile serving as an upper bound for the residuals indexed by $S$. Thus the proof of Corollary \ref{cor:expectedcorrupterrbound} follows verbatim from \cite[Lemma~2]{qRK} with $Q_1$ in place of $Q$. We note here that using $Q_0$ would instead be a means to lower bound the left hand side of ({\ref{eq:expectedcorrupterrboundqrk}). 
        
        \begin{lem}
            In the same setting as Corollary \ref{cor:Q1boundcor}, we have
            \[
                \E_{i\in B\sm S}\|x_{k+1}-\x\|^2\leq \groupp{1-\frac{\s_{q_1-\beta,\min}^2(A)}{q_1 m}-\frac{\s_{q_0-\beta,\min}^2(A)}{q_0q_1 m^2}}\|x_k-\x\|^2.
            \]
            \label{lem:lemdqrk}
        \end{lem}
        \begin{proof}
            We note that the collection of indices $B\sm S$ are those that are both admissible and non-corrupt, so $A_{B\sm S}\x = b_t$. We note that $B_1\sm (S\cup S_0)$ is $B\sm S$ combined with the non-corrupt indices corresponding to residuals smaller than the $q_0$-quantile $Q_0$. This suggests we might be able to use the results of \texttt{rqRK} with a well-curated quantile parameter. Unfortunately, we cannot directly apply the results of \texttt{rqRK} because the hypotheses of Theorem \ref{theorem:rqrk} requires $x_k$ to satisfy 
            \begin{equation}
                \lll x_k, a_j\rrr = b_j
                \label{eq:rowsatisfied}
            \end{equation}
            for some $j\in B_1\sm (S\cup S_0)$. However, it is entirely possible that the index selected in iteration $k$ belongs to $C$, whereby it is not clear that there exists an $i\in B_1\sm (S\cup S_0)$ for which (\ref{eq:rowsatisfied}) holds. Instead, we will provide a modification of the proof of Theorem \ref{theorem:rqrk} that allows for the possibility that $x_k$ was generated using a corrupt row. 
        
            Let $r$ be the index of the row used to generate $x_k$. We will let $V$ be the matrix $A$ indexed by $\big(B_1\sm (S\cup S_0)\big)\cup \{r\}$. We denote the $i$-th row of $V$ by $v_i$. Let $\tau$ be a map such that $v_{i}=a_{\tau(i)}$, and let $r' = \tau\inv(r)$ so that $v_r'=a_r$. To stay consistent with the notation from Theorem \ref{theorem:rqrk}, we let $X_{k+1}$ and $Y_{k+1}$ be the random variables that assume an indices $\tau\inv\groupb{(B_1\sm (S\cup S_0))\cup \{r\}}$ and $\tau\inv\groupb{B\sm S}$ with uniform distribution at step $k+1$, respectively.
            
            Note that the relation 
            \[
                \|x_{k+1}-\x\|^2=\groupp{1-\l|\l\lll \frac{x_{k}-\x}{\|x_{k}-\x\|}, v_{Y_k}\r\rrr\r|^2}\|x_{k}-\x\|^2,
            \]
            from the proof of Theorem \ref{theorem:rqrk} no longer holds for any $x_k$ because it relies on $b_i-\lll \x, v_j\rrr = 0$, and the possibility of $r$ corresponding to a corrupted row means we cannot ensure this holds for $j=r'$. However, we can still use 
            \[
                \|x_{k+1}-\x\|^2=\|x_{k}-\x\|^2-\groupabs{b_{\tau(Y_{k+1})} - \lll x_{k}, v_{Y_{k+1}}\rrr}^2.
            \]
            We will focus on bounding 
            \[
                \E\groupb{\groupabs{b_{\tau(Y_{k+1})} - \lll x_{k}, v_{Y_{k+1}}\rrr}^2}.
            \]
        
            We define $\ell^\ast=\|V\|_F^2$ and $\ell=\groupabs{(B_1\sm(S\cup S_0)\cup \{r\}}-\groupabs{B\sm S}$. With this, we let $f_{k+1}=g_{k+1}\circ \s$, where $g_{k+1}(j)=\groupabs{b_{\tau(j)}-\l\lll x_k,v_j\r\rrr}^2$ and $\s$ is a permutation of $[\ell^\ast]$ such that $g_{k+1}(\s(1)),\dots, g_{k+1}(\s(\ell^\ast))$ is non-decreasing. Corollary \ref{cor:maincorollary} implies 
            \begin{align*}
                \E\groupb{\groupabs{b_{\tau(Y_{k+1})} - \lll x_{k}, v_{Y_{k+1}}\rrr}^2} &\geq \E\groupb{\groupabs{b_{\tau(X_{k+1})} - \lll x_{k}, v_{X_{k+1}}\rrr}^2} \\
                &+ \sum_{j=1}^{\ell} \Big(f_{k+1}(j+1)-f_{k+1}(j)\Big)\PP\groupp{X_{k+1}=j\mid X_{k+1}\not\in [j-1]}.
            \end{align*}
            Immediately, we have $f_{k+1}(1)=0$ because $x_{k}$ satisfies $b_{r}-\lll x_{k},v_{r'}\rrr=0$. It is also easy to see $f_{k+1}(\ell+1)\geq Q_0^2$. Combining these facts with 
            \[
                \PP(X_{k+1}=j\mid X_{k+1}\in[j-1])\geq \PP(X_{k+1}=j) = \frac{1}{\ell^\ast},
            \]
           and we have 
            \[
                \E\groupb{\groupabs{b_{\tau(Y_{k+1})} - \lll x_{k}, v_{Y_{k+1}}\rrr}^2} \geq \E\groupb{\groupabs{b_{\tau(X_{k+1})} - \lll x_{k}, v_{X_{k+1}}\rrr}^2} + \frac{Q_0^2}{\ell^\ast}.
            \]
            We now turn our focus to $\E\groupb{\groupabs{b_{\tau(X_{k+1})} - \lll x_{k}, v_{X_{k+1}}\rrr}^2}$. In the case that $r\in C$, let $J$ denote the set of rows of $V$ not given by $a_r$ in our construction of $V$. If $r\not\in C$, we let $J$ be all the indices of the rows of $V$. In either case, we have $b_{\tau(j)}-\lll x_{k},v_j\rrr = \lll \x - x_{k}, v_j\rrr$ for all $j\in J$. Then whether $r\in C$ or $r\not\in C$, we have 
            \begin{align*}
                \E\groupb{\groupabs{b_{X_{k+1}} - \lll x_{k}, v_{X_{k+1}}\rrr}^2} &= \sum_{j=1}^{\ell^\ast} \frac{1}{\ell^\ast} \groupabs{b_j - \lll x_{k}, v_j\rrr}^2 \\
                &\geq \sum_{j\in J} \frac{1}{\ell^\ast} \groupabs{\lll \x- x_k,v_j\rrr}^2 \\
                &= \frac{1}{\ell^\ast} \|V_J (x_k-\x)\|^2 \\
                &\geq \frac{\s_{\min}^2(V_J)\|x_k-\x\|^2}{\ell^\ast}.
            \end{align*}
            Because $J$ is the set of indices of non-corrupt rows with residual entries less than $Q_1$ in magnitude, it is clear $(q_1-\beta)m\leq |J|$. This immediately implies $\s_{q_1-\beta,\min}(A) \leq \s_{\min}(V_J)$. Moreover, $V$ consists only of rows whose residuals are less than $Q_1$, so $\ell^\ast\leq q_1m$.  Altogether, we have
            \[
                \E_{i\in B\sm S}\groupb{\|x_{k+1}-\x\|^2} \leq \groupp{1-\frac{\s_{q_1-\beta,\min}^2(A)}{q_1m} - \frac{Q_0^2}{q_1m\|x_k-\x\|^2}} \|x_k-\x\|^2.
            \]
        
            Moreover, we can denote $N=\groupc{i\in [m]: i\not\in C, |b_j-\lll x_k, a_j\rrr|\leq Q_0}$. There are at least $(q_0-\beta)m$ rows that are not violated that witness $|b_j-\lll x_k, a_j\rrr|\leq Q_0$, so $|N|\geq (q_0-\beta)m$. It is also easy to see $|N|\leq q_0 m$. For $j\in N$, we have $\groupabs{b_j-\lll x_k,a_j\rrr} = \groupabs{\lll x_k-\x,a_j\rrr}$, so  
            \begin{align*}
                q_0m Q_0^2 &\geq \sum_{j\in N} \groupabs{\lll x_k-\x, a_j\rrr}^2 \\
                &= \|A_N (x_k-\x)\|^2 \\
                &\geq \s_{\min}^2(A_N)\|x_k-\x\|^2 \\
                &\geq \s_{q_0-\beta,\min}^2(A)\|x_k-\x\|^2.
            \end{align*}
            In conclusion,  
            \[
                \E_{i\in B\sm S}\groupb{\|x_{k+1}-\x\|^2} \leq \groupp{1-\frac{\s_{q_1-\beta,\min}^2(A)}{q_1m} - \frac{\s_{q_0-\beta,\min}^2(A)}{q_0q_1m^2}}\|x_k-\x\|^2.
            \]
        \end{proof}
        
        \noindent Just as before, if $(q_0-\beta)m<n$, we have $\s_{q_0-\beta,\min}(A)=0$ so that our bound simplifies to that of \cite[Lemma 3]{qRK}. But, as before, we have that $\E_{i\in B\sm S}\|x_{k+1}-\x\|<\E_{i\in B_1\sm (S\cup S_0)}\|x_{k+1}-\x\|^2$ as long as there are two distinct values in 
        \[
            \groupc{\l|\l\lll \frac{x_{k}-\x}{\|x_k-\x\|}, a_j\r\rrr\r|^2}_{j\in B_1\sm(S\cup S_0)}.
        \]
    
    \subsection{The Bound on \texttt{dqRK}}
        \label{sec:boundondqrk}
        Theorem \ref{theorem:dqrk} asserts that if the conditioning on $A$ in (\ref{eq:syseq}) is sufficiently well behaved and $b$ has sparse, unbounded corruption, it is possible to recover a solution to the true system $Ax=b_t$ faster than possible with \texttt{qRK} using only information from the corrupted system. The only step where the rate of convergence might not exceed \texttt{qRK} is when producing the first iteration $x_1$. Just as in Theorem \ref{theorem:rqrk}, the hypothesis that $\lll x_0,a_i\rrr = b_i$ for some $i\in[m]$ does not change the convergence result drastically. In practice, it would be feasible to choose an index $i\in [m]$ deterministically or at random to define the first iterate $x_0=b_ia_i$ to satisfy the hypothesis. We closely mimic Steinerberger's \cite{qRK} proof with our modified 3rd lemma in order to attain Theorem \ref{theorem:dqrk}:
        \dqrkresult*
        \begin{proof}
            Let $e_k=x_k-\x$. Starting with 
            \begin{align*}
                \E\|e_{k}\|^2 &= \E_{i\in B\sm S} \|e_{k}\|^2 \PP(i\in B\sm S) + \E_{i\in S} \|e_{k}\|^2 \PP(i\in S) \\
                &=  \groupp{1-\frac{|S|}{(q_1-q_0)m}}\E_{i\in B\sm S} \|e_{k}\|^2  + \frac{|S|}{(q_1-q_0)m} \E_{i\in S} \|e_{k}\|^2
            \end{align*}
            and then applying Corollary \ref{cor:expectedcorrupterrbound} and Lemma \ref{lem:lemdqrk}, we have 
            \begin{align*}
                \E\|e_{k}\|^2 &\leq \Bigg[\groupp{1-\frac{|S|}{(q_1-q_0)m}}\groupp{1-\frac{\s_{q_1-\beta,\min}^2(A)}{q_1 m}-\frac{\s_{q_0-\beta,\min}^2(A)}{q_0q_1 m^2}}\\
                &+ \frac{|S|}{(q_1-q_0)m}\groupp{1+\frac{\s_{\max}^2(A)}{\sqrt{|S|}\sqrt{m}}\groupp{\frac{2}{\sqrt{1-q_1-\beta}}+\frac{\sqrt{\beta}}{1-q_1-\beta}}}\Bigg]\|e_{k-1}\|^2.
            \end{align*}
            It is clear this is increasing in $|S|$, so $|S|\leq \beta m$ gives us 
            \begin{align*}
                \E\|e_{k}\|^2 &\leq \Bigg[\groupp{1-\frac{\beta}{q_1-q_0}}\groupp{1-\frac{\s_{q_1-\beta,\min}^2(A)}{q_1 m}-\frac{\s_{q_0-\beta,\min}^2(A)}{q_0q_1 m^2}}\\
                &+ \frac{\beta}{q_1-q_0}\groupp{1+\frac{\s_{\max}^2(A)}{m\sqrt{\beta}}\groupp{\frac{2}{\sqrt{1-q_1-\beta}}+\frac{\sqrt{\beta}}{1-q_1-\beta}}}\Bigg]\|e_{k-1}\|^2.
            \end{align*}
            We can express the coefficient on $\|e_{k-1}\|^2$ as 
            \[
                1+\frac{\s_{\max}^2(A)}{(q_1-q_0)m}\groupp{\frac{2\sqrt{\beta}}{\sqrt{1-q_1-\beta}}+\frac{\beta}{1-q_1-\beta}}-\groupp{q_1-q_0-\beta}\groupp{\frac{\s_{q_1-\beta,\min}^2(A)}{(q_1-q_0)q_1 m}+\frac{\s_{q_0-\beta,\min}^2(A)}{(q_1-q_0)q_0q_1 m^2}}
            \]
            The sum above is strictly greater than 1 if $q_1-q_0\leq \beta$, so we must assume otherwise. Then the following is sufficient to ensure decay:
            \[
                \frac{q_1}{q_1-q_0-\beta}\groupp{\frac{2\sqrt{\beta}}{\sqrt{1-q_1-\beta}}+\frac{\beta}{1-q_1-\beta}} < \frac{1}{\s_{\max}^2(A)}\groupp{\s_{q_1-\beta,\min}^2(A)+\frac{\s_{q_0-\beta,\min}^2(A)}{q_0m}}.
            \]
            These results hold when the first $k$ iterates are fixed. By taking the total expectation and proceeding by induction, we have our desired result. \\
        \end{proof}
        
        It might be interesting to note that the \texttt{qRK} bound can be recovered by Theorem \ref{theorem:dqrk} in some sense. If we follow the convention $\s_{q,\min}(A)=0$ whenever $qm<1$, letting $q_0\ra 0$ gives us the original bound for \texttt{qRK}.
        
        We remark that we can actually choose $q_0$ and $q_1$ so that Algorithm \ref{alg:dqrk} always chooses the row corresponding to the $k$-th largest residual. This can be achieved by letting $q_0=\frac{k-1}{m}$ and $q_1=\frac{k}{m}$. In order for Theorem \ref{theorem:dqrk} above to be applicable in this scenario, we require $\beta< \frac{1}{m}$. In other words, we have no corruption, so we may set $\beta=0$. Then so long as $\s_{q_1,\min}(A)>0$, the theorem above ensures picking the row with the $k$-th largest residual in a system without error will yield a convergent sequence of iterates. More specifically, we can write
        \begin{equation}
            \E\|x_{k}-\x\|^2\leq \groupp{1-\frac{\s_{k/m,\min}^2(A)}{k}-\frac{\s_{(k-1)/m,\min}^2(A)}{k^2-k}}^k \|x_0-\x\|^2 \label{eq:kres}
        \end{equation}
        Though this might be useful in scenarios where corruption can be controlled, choosing rows with larger residuals does not bode well in general inconsistent systems. We will later see numerical experiments where an algorithm like \texttt{dqRK} does not perform in line with the bound in Theorem \ref{theorem:dqrk} because the matrix fails to meet the sufficient conditions. However, we do notice that iterates from \texttt{dqRK} appear to approach an error horizon to the least squares solution just like the iterates of \texttt{qRK} or \texttt{RK} might. Work has been done relating \texttt{RK} or variations of \texttt{RK} to solving for the least-norm least-square solution to $Ax=b$ arising in the case where $A$ and $b$ might be noisy~\cite{needell_LS, Zouzias_REK, exactMSE, doublynoisysystem}.

\section{Numerical Results}
    While theoretically, the new bounds found for \texttt{rqRK} and \texttt{dqRK} are not drastically different from their analogs without lower quantiles, experimentally, the proposed algorithms vastly outperform the old. In addition, \texttt{dqRK} has almost no cost performance-wise when compared to \texttt{qRK}.

    All code was written in \texttt{MATLAB} on an ASUS Zephyrus g14 laptop with 16 GB of memory, an AMD Ryzen 9 5900HS CPU clocked at 3.30 GHz, and a dedicated laptop RTX 3060 GPU. Reproducible code can be found at \url{https://github.com/Mr-E-User/Reverse-Quantile-RK-and-its-Application-to-Quantile-RK.git}. 

    \subsection{Numerical Results for \texttt{rqRK}}
        \begin{figure}[!htb]
        \centering
            \input{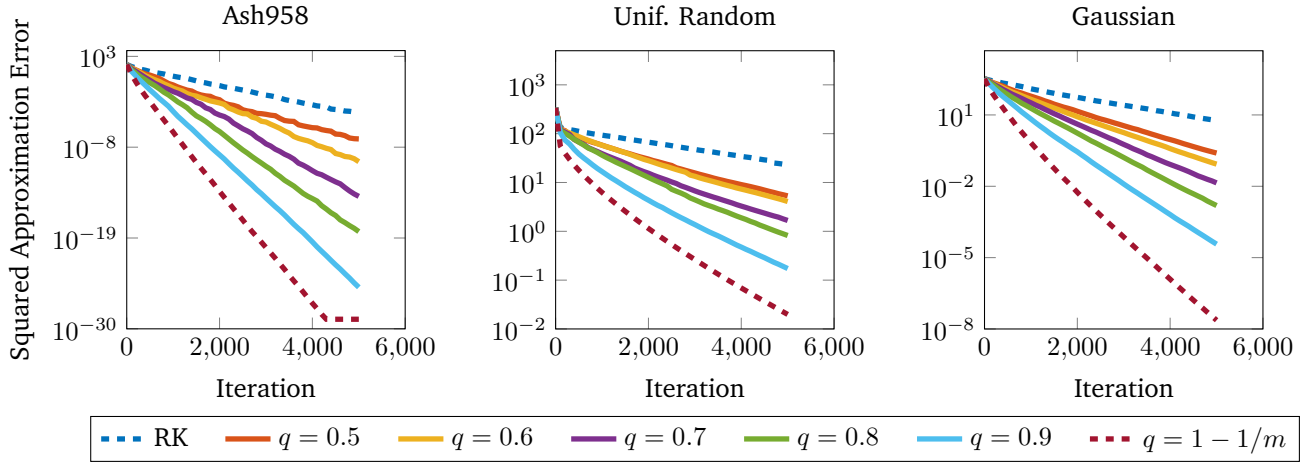}
            \caption{Comparison between sample runs of \texttt{RK} and \texttt{rqRK} for varying lower quantiles $q$.}
            \label{fig:rqRK_comp}
        \end{figure}
        
        Figure~\ref{fig:rqRK_comp} compares the performance of \texttt{RK} and \texttt{rqRK} with different choices of $q$ on the $958\times 292$ \texttt{ash958} matrix from SparseSuite, a $5000\times 1000$ uniformly random matrix, and a $5000\times 1000$ Gaussian matrix \cite{sparsesuite}. In each case, the rows were normalized. The quantiles in Figure \ref{fig:rqRK_comp} were chosen by starting at $q=\frac{1}{2}$ and then increased by $0.1$ to obtain the next quantile except at the largest $q$ value where we chose $\frac{m-1}{m}$ specifically. Our first iterate $x_0$ was initialized at the origin. The solution $\x$ was generated at random and was used to produce $b$. 
        
        Unsurprisingly, the larger the utilized quantile is, the better \texttt{rqRK} performs, as is demonstrated in Figure \ref{fig:rqRK_comp}. It should be noted that not pictured here are the cases when $q<\frac{1}{2}$. In these instances, it appears \texttt{rqRK} still typically outperforms \texttt{RK}; however, it becomes more difficult to distinguish the behaviors of two error plots for close quantiles $q, q' \ll \frac{1}{2}$. Interestingly, there is rarely overlap in the approximation error plots when $q$ is sufficiently close to 1; by extension, the Motzkin method, which corresponds to the dashed red line, appears to be a lower bound for all the approximation errors for $q<\frac{m-1}{m}$. 
    
    \subsection{Numerical Results for \texttt{dqRK}}

        \begin{figure}[!htb]
            \centering
            \input{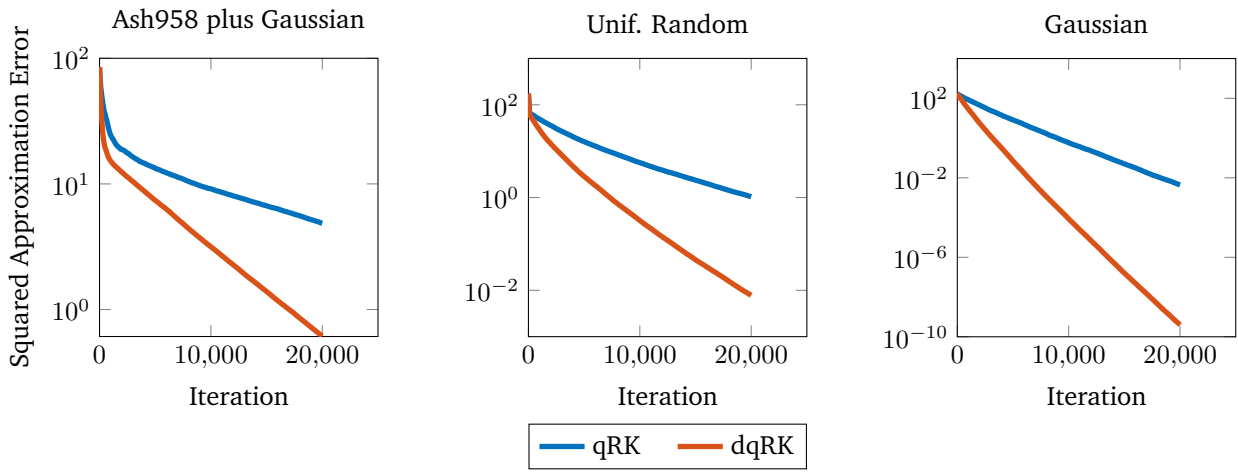}
            \caption{Comparison between sample runs of \texttt{qRK} and \texttt{dqRK} for $q_0=0.6$, $q_1=0.8$ on a linear system with sparse noise \eqref{eq:corruptsyseq} with $\beta = 0.05$.}
            \label{fig:dqRK_comp}
        \end{figure}
        
        Figure~\ref{fig:dqRK_comp} compares \texttt{qRK} and \texttt{dqRK} for $q_0=0.6$, $q_1=0.8$, and $\beta = 0.05$ on the \texttt{ash958} matrix from SparseSuite plus a Gaussian matrix scaled down by a factor of 100, a $2500\times 500$ uniformly random matrix, and a $2500\times 500$ Gaussian matrix to solve sparse noisy systems \cite{sparsesuite}. The corruption vector $b_\ve$ was generated by uniformly randomly choosing the indices that would receive non-zero entries randomly chosen between 0 and 1. Using a randomly generated solution $\x$, we set $b_t= A\x$ and $b=b_\ve+b_t$ to create the corrupt system of equations 
        \begin{equation}
            Ax = b. \label{eq:corruptsyseq}
        \end{equation}
        In each case, the rows were normalized. 
        
        Just as \texttt{rqRK} accelerated \texttt{RK}, we see experimental results in Figure \ref{fig:dqRK_comp} indicating \texttt{dqRK} has a similar effect on \texttt{qRK}. Perhaps more importantly, we also observe experimentally in Table \ref{table:algruntimes} that the performance cost to implement \texttt{dqRK} is almost identical to that of \texttt{qRK}. The most expensive computation is determining the terms $\groupabs{\lll x_k,a_j\rrr -b_j}$ and the corresponding quantiles. However, \texttt{qRK} already requires the computation of the $\groupabs{\lll x_k,a_j\rrr -b_j}$ terms, so utilizing them in \texttt{dqRK} adds no computation cost. As for the quantile computations, it is evident finding a second quantile adds a minimal amount of processing time to \texttt{dqRK} due to \texttt{MATLAB}'s specific method of computing quantiles. As noted in their documentation, a sorting-based algorithm is used, so only one ``sort'' is required to compute one or more quantiles. This is highlighted in Table~\ref{table:algruntimes} where runs of 1000 iterations was implemented for each matrix size. For each matrix, $q_0=0.6$, $q_1=0.8$, and $\beta =0.05$ were used, and the solution and corruption in $b$ is generated the same way as in Figure \ref{fig:dqRK_comp}. Matlab's \texttt{timeit} function was used to measure the time taken to run the 1000 iterations. To further reinforce this point, Table \ref{table:thresholdtimes} displays the times it takes \texttt{qRK} and \texttt{dqRK} with Gaussian matrices to reach an error threshold. More specifically, for each matrix dimension, a Gaussian matrix of the specified size is generated, $b$ is generated the same way as in Figure \ref{fig:dqRK_comp}, and the times to reach the threshold squared approximation error $10^{-8}$ by \texttt{qRK} and \texttt{dqRK} using parameters $q_0=0.6$ and $q_1=0.8$ are recorded. The times were again measured using Matlab's \texttt{timeit} function.
        
        \begin{table}[!htb]
            \centering
            \begin{tabular}{ |p{3cm}||p{2.7cm}|p{2.7cm}|}
                \hline
                \multicolumn{3}{|c|}{Algorithm Run Times}                                       \\
                \hline
                Matrix Dimension        & \texttt{qRK} Time (s)     &  \texttt{dqRK} Time (s)   \\
                \hline
                $1000\times 100$        & 0.092                     & 0.099                     \\
                $1000\times 500$        & 0.192                     & 0.211                     \\
                $5000\times 100$        & 0.290                     & 0.283                     \\
                $5000\times 500$        & 0.786                     & 0.793                     \\
                $5000\times 1000$       & 1.481                     & 1.460                     \\
                $10000\times 1000$      & 2.967                     & 2.964                     \\
                $50000\times 1000$      & 15.308                    & 15.202                    \\
                $100000\times 1000$     & 30.446                    & 29.961                    \\
                \hline
            \end{tabular}
            \caption{The matrices of each listed size are generated with uniformly random entries. Times are rounded to the nearest millisecond.}
            \label{table:algruntimes}
        \end{table}
        
        \begin{table}[!htb]
            \centering
            \begin{tabular}{ |p{3cm}||p{2.7cm}|p{2.7cm}|}
                \hline
                \multicolumn{3}{|c|}{Time to Reach Threshold}                                       \\
                \hline
                Matrix Dimension        & \texttt{qRK} Time (s)         &  \texttt{dqRK} Time (s)   \\
                \hline
                $1000\times 100$        & 0.509                         & 0.211                     \\
                $1000\times 500$        & 109.667                       & 32.281                    \\
                $5000\times 100$        & 1.237                         & 0.503                     \\
                $5000\times 500$        & 24.021                        & 9.026                     \\
                $5000\times 1000$       & 151.564                       & 53.479                    \\
                $10000\times 1000$      & 192.809                       & 70.784                    \\
                $50000\times 1000$      & 801.50                        & 309.11                    \\
                $100000\times 1000$     & 1473.77                       & 569.09                    \\
                \hline
            \end{tabular}
            \caption{The matrices used are Gaussian matrices with the indicated dimension. Times are rounded to the nearest millisecond.}
            \label{table:thresholdtimes}
        \end{table}
        
        Figure \ref{fig:dqRK_vary_quantile} provides insight into the relative effect of each quantile parameter. The corruption vector $b_\ve$ was generated using the same method outlined for Figure \ref{fig:dqRK_comp}. Lines of the same color utilize the same value for $q_0$, and lines of the same style utilize the same value for $q_1$. Figure \ref{fig:dqRK_vary_quantile} suggests using a larger upper or lower quantile noticeably increases the rate of convergence. This should be expected, as increasing either $q_0$ or $q_1$ increases the likelihood \texttt{dqRK} will generate iterates with larger distances between successive points. We point out however, this is not without a caveat, as the convergence of these runs is heavily influenced by $q_1<1-\beta$. As such, implementation of \texttt{dqRK} might require a balance to be struck between the quantile values, which influences the convergence rate, and the size of $\beta$, the proportion of corrupted entries allowed. 
        
        \begin{figure}[!htb]
            \centering
            \input{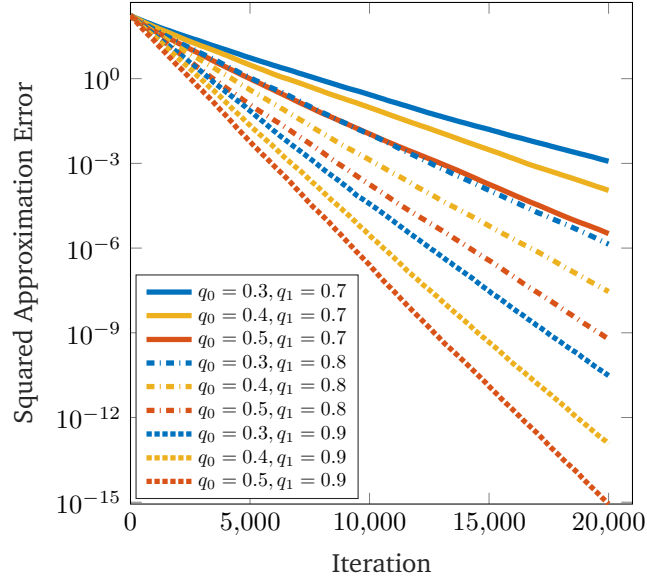}
            \caption{Comparison between sample runs of \texttt{dqRK}, using varying quantile parameters. The error proportion parameter $\beta$ is fixed at $0.05$.}
            \label{fig:dqRK_vary_quantile}
        \end{figure}

\section{Final Remarks}
    \label{sec:final_rem}
    The last comment is a note on the failure of \texttt{dqRK} on some sparse matrices where sparse corruption in present in $b$. Oftentimes, if a matrix is too sparse, we might witness behavior different from that displayed in Figure \ref{fig:dqRK_comp}. Despite not converging to the exact solution $\x$, it does appear to hone in on an error horizon smaller than that which might be accomplished by standard \texttt{RK} as outlined in \cite{needell_LS}. Moreover, a similar pattern is present with \texttt{qRK}, but this error horizon is approached notably faster in the \texttt{dqRK} scheme. This information is encapsulated in Figure \ref{fig:error_horizon}. Empirically, this scenario has been repeated with other sparse matrices with appropriate quantile values. We leave the further exploration of this phenomena for future work.

    \begin{figure}[!htb]
        \centering
        \input{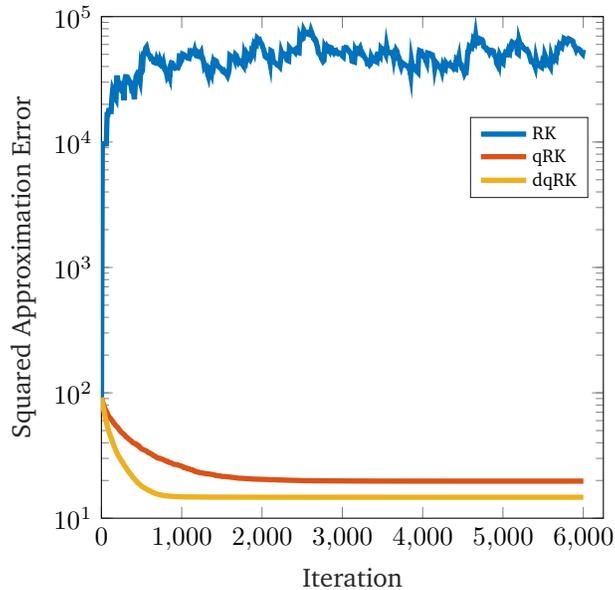}
        \caption{Comparison between a sample run of \texttt{RK}, \texttt{qRK}, and \texttt{dqRK} for $q_0=0.6$, $q_1=0.8$, and $\beta = 0.05$ on the row-normalized \texttt{ash958} matrix from SparseSuite \cite{sparsesuite}.}
        \label{fig:error_horizon}
    \end{figure}
    
    The corrupted $b$ and solution $\x$ in Figure \ref{fig:error_horizon} were generated just as in Figure \ref{fig:dqRK_comp}, with the exception that the corruption present in $b$ was scaled by a factor of 100. To further explore this failure of convergence, Table \ref{table:illcondition} presents entries of at most 5 significant figures where we report $\s_{1-1/m,\min}(A)$ and 
    \begin{equation}
        \begin{split}
            E_{q_0,q_1,\beta}(A)&=\frac{1}{\s_{\max}^2(A)}\groupp{\s_{1-1/m,\min}^2(A)+\frac{\s_{1-1/m,\min}^2(A)}{q_0m}} \\
            &\quad-\frac{q_1}{q_1-q_0-\beta}\groupp{\frac{2\sqrt{\beta}}{\sqrt{1-q_1-\beta}}+\frac{\beta}{1-q_1-\beta}}.
        \end{split}
        \label{eq:E} 
    \end{equation} 
    Notably, the matrices from SpareSuite used in Table \ref{table:illcondition} are row-normalized prior to the computation of these \cite{sparsesuite}. 
    
    Given that $\s_{\alpha,\min}(A)$ is increasing in $\alpha$, it is easy to see that anytime the $E_{q_0,q_1,\beta}(A)$ entry in Table \ref{table:illcondition} is negative, the corresponding matrix does not satisfy the sufficient condition for \texttt{dqRK} to converge. Namely, if $E_{q_0,q_1,\beta}(A)\leq 0$, we have 
    \begin{align*}
        \frac{q_1}{q_1-q_0-\beta}\groupp{\frac{2\sqrt{\beta}}{\sqrt{1-q_1-\beta}}+\frac{\beta}{1-q_1-\beta}} &\geq \frac{1}{\s_{\max}^2(A)}\groupp{\s_{1-1/m,\min}^2(A)+\frac{\s_{1-1/m,\min}^2(A)}{q_0m}} \\
        &\geq \frac{1}{\s_{\max}^2(A)}\groupp{\s_{q_1-\beta,\min}^2(A)+\frac{\s_{q_0-\beta,\min}^2(A)}{q_0m}}.
    \end{align*}
    These matrices are sparse, so the failure to meet the sufficient condition can heuristically be thought of as lacking redundant information. That is, some information might be encoded solely in a small batch of rows, which means any time corruption occurs in these rows, we risk losing too much information to recover our desired $x$ vector. 
    
    \begin{table}[!htb]
        \centering
        \begin{tabular}{ |p{4cm}||p{4cm}|p{4cm}|}
            \hline
            \multicolumn{3}{|c|}{Ill-Conditioned Matrices}                                  \\
            \hline
            Matrix $A$              &  $\s_{1-1/m,\min}(A)$     & $E_{q_0,q_1,\beta}(A)$    \\
            \hline
            \texttt{ash958}         & 0.7392                    & -3.3403                   \\
            \hline
            \texttt{photogrammetry} & 1.1907e-08                & -3.4012                   \\
            \hline 
            \texttt{well1033}       & 2.9906e-18                & -3.4012                   \\
            \hline 
            \texttt{illc1033}       & 1.0791e-18                & -3.4012                   \\
            \hline
            \texttt{WorldCities}    & 0.1675                    &  -3.4010                  \\ 
            \hline 
            \texttt{ash608}         & 0.5616                    & -3.3612                   \\
            \hline
        \end{tabular}
        \caption{Values $\s_{1-1/m,\min}(A)$ and $E_{q_0,q_1,\beta}(A)$ for varying sparse matrices. Matrices were pulled from SparseSuite, and the rows of each matrix were normalized \cite{sparsesuite}. The parameters used were $q_0=0.6$, $q_1=0.8$, and $\beta=0.05$. }
        \label{table:illcondition}
    \end{table}

\bibliographystyle{plain}
\bibliography{refs.bib}

\end{document}